\documentclass[a4paper,12pt]{amsart}
\usepackage{ifthen}
\usepackage{graphicx}
\usepackage{mathrsfs}
\usepackage{color}
\nonstopmode
\numberwithin{equation}{section}
\usepackage[section]{placeins}

\newtheorem{thm}{Theorem}[section]
\newtheorem{theorem}{Theorem}

\newtheorem{cor}[thm]{Corollary}
\newtheorem{lem}[thm]{Lemma}

\newtheorem{rem}{Remark}

\theoremstyle{definition}

\newtheorem{eg}[thm]{Example}

\newenvironment{pf}[1][]{%
 \vskip 3mm
 \noindent
 \ifthenelse{\equal{#1}{}}%
  {{\slshape Proof. }}%
  {{\slshape #1.} }%
 }%
{\qed\bigskip}

\newcounter{alphabet}

\newcommand{\C}{{\mathbb C}}
\newcommand{\D}{{\mathbb D}}
\newcommand{\es}{{\mathcal S}}
\newcommand{\T}{{\mathbb T}}

\newcommand{\har}{{\mathcal{H}}}
\newcommand{\harm}{{\mathrm{H}}}
\newcommand{\seq}{{(\{\varphi_n\},\{\psi_n\})}}

\newcommand{\sh}{{\mathcal S}_\harm}
\newcommand{\hsst}{{\mathcal {SS}}_\harm}

\newcommand{\st}{{\mathcal{S}^*}}

\newcommand{\bD}{{\overline{\mathbb D}}}

\newcommand{\sphere}{{\widehat{\mathbb C}}}

\renewcommand{\Re}{{\,\operatorname{Re}\,}}

\newcommand{\sumo}{{\sum_{n=0}^{\infty}}}
\newcommand{\suma}{{\sum_{n=1}^{\infty}}}
\newcommand{\sumb}{{\sum_{n=2}^{\infty}}}
\newcommand{\Sig}{{\Sigma_\harm}}
\newcommand{\Siga}{{\Sigma'_\harm}}
\newcommand{\Sigb}{{\Sigma''_\harm}}

\makeatletter
\@namedef{subjclassname@2020}{\textup{2020} Mathematics Subject Classification}
\makeatother

\newcounter{minutes}
\divide\time by 60
\newcounter{hours}
\multiply\time by 60 \addtocounter{minutes}{-\time}

\begin{document}
\bibliographystyle{amsplain}
\title[Q.c. Exts. of Harm. Univ. Maps. of the Unit Disk]{Quasiconformal Extensions of Harmonic Univalent Mappings of the Unit Disk}
\def\thefootnote{}
\footnotetext{
\texttt{\tiny File:~\jobname .tex,
          printed: \number\year-\number\month-\number\day,
          \thehours.\ifnum\theminutes<10{0}\fi\theminutes}
}
\makeatletter\def\thefootnote{\@arabic\c@footnote}\makeatother
\author[X.-S. Ma]{Xiu-Shuang Ma}
\address{School of Mathematics and Statistics\\
Hunan University of Science and Technology\\
Xiangtan, 411201, China}
\email{maxiushuang@gmail.com}

\begin{abstract}
This note examines sufficient conditions for the quasiconformal extendibility of harmonic mappings defined in the unit disk.
It is demonstrated that a harmonic strongly starlike mapping admits a quasiconformal extension to the entire plane, and an explicit formulation of its extension function is provided. 
Additionally, the quasiconformal extendibility of harmonic mappings defined in the exterior of the unit disk is explored.
%
\end{abstract}


\maketitle

\section{Introduction}

Let $\D=\{z:|z|<1\}$ denote the unit disk in the complex plane $\C$. 
Let $\T$ and $\bD$ be the unit circle $\{z:|z|=1\}$ and the closed unit disk $\{z:|z|\le1\}$, respectively. 
A complex-valued function $f(z)=u(z)+iv(z)$ defined in a domain $\Omega\subset\C$ is harmonic if $u$ and $v$ are both real harmonic for $z=x+iy\in\Omega$. 
If $\Omega$ is simply-connected, then any harmonic function $f$ has a canonical decomposition $f=h+\overline{g}$, where $h$, $g$ are analytic functions on $\Omega$.

Consider a class of harmonic functions $f$ defined on $\D$ with $f(0)=f_z(0)-1=0$, denoted by $\har$. 
Each function $f=h+\overline{g}\in\har$ has the power series expansions for $h$ and $g$ by
\begin{equation}\label{series}
h(z)=z+\sum_{n=2}^{\infty}a_n z^n,\quad g(z)=\sum_{n=1}^{\infty}b_n z^n,\quad z\in\D.
\end{equation}
The Jacobian of $f$ is given by 
$$
J_f(z)=|f_z(z)|^2-|f_{\bar z}(z)|^2=|h'(z)|^2-|g'(z)|^2,\quad z\in\D.
$$ 
A well-known result asserts that $f$ is locally univalent (one-to-one) if and only if $J_f(z)\neq 0$ for  $z\in\D$. 
If $J_f(z)>0$ (or $J_f(z)<0$), then $f$ is said to be {\it sense-preserving (or sense-reversing)}. 

A sense-preserving harmonic mapping $f$ is said to be {\it $k$-quasiconformal~($k$-q.c.)}, if $f$ has the complex dilatation $\mu_f(z)=f_{\bar{z}}(z)/f_z(z)$, with $|\mu_f(z)|\leq k<1$ for almost every point in $\Omega$. 
In the most literature, $f$ is known as a {\it $K$-quasiconformal} mapping with $K=(1+k)/(1-k)\ge1$. 
See \cite{ahlfors06, LV73} for basic theory of quasiconformal mappings. 

Let $\es$ denote a class of functions $h$ analytic and univalent in $\D$, normalized by $h(0)=h'(0)-1=0$, and $\st(\alpha)~(0\le\alpha<1)$ be its subclass of functions satisfying
$$
\left|\arg{\frac{zh'(z)}{h(z)}}\right|\le\frac{\pi\alpha}2,\quad z\in\D, 
$$
named as {\it strongly starlike functions}. 
A result given by Fait, et al. \cite{FKZ76} shows that 
each $f\in\st(\alpha)$ has a quasiconformal extension to the whole plane.
It is also proved that $f\in\es$ satisfying $\sumb n|a_n|\le k<1$ is strongly starlike, and hence quasiconformally extendible on the plane, with $|\mu_f|\le k$ a.e. in $\C$. 

Now it is natural to ask if the same results hold for harmonic strongly starlike functions. 
To answer the question, we first introduce the notion of harmonic hereditarily strongly starlike functions.
A subclass of $\har$ consisting of functions which are sense-preserving and univalent in $\D$ is denoted by $\sh$. 
This class has been investigated by Clunie and Sheil-Small \cite{CS84} (see also Duren \cite{D04}).
For $0<\alpha\le1$, the class of functions $f\in\sh$ is said to be {\it harmonic hereditarily strongly starlike of order $\alpha$} if it satisfies
$$
\left|\arg{\frac{zf_z(z)-\overline{z}f_{\overline{z}}(z)}{f(z)}}\right|<\frac{\pi\alpha}{2},\quad z\in\D\backslash\{0\},
$$ 
denoted by $\hsst(\alpha)$ (see \cite{MPS21}). 
It leads to the inequality 
$$
\Re\left({\frac{zf_z(z)-\overline{z}f_{\overline{z}}(z)}{f(z)}}\right)>0,\quad z\in\D\backslash\{0\},
$$
if $\alpha=1$.
Then $f$ is harmonic fully starlike. Here, harmonic full starlikeness (harmonic full convexity mentioned below) means that $f$ is harmonic starlike (convex) and maps each disk $r\D\ (0<r<1)$ univalently onto a starlike (convex) domain(see \cite{CDO04}). 

Let $\{\varphi_n\}_{n=2,3,...}$ and $\{\psi_n\}_{n=1,2,...}$ be two sequences of non-negative real numbers. We denote by $\har(\{\varphi_n\},\{\psi_n\})$ the class of harmonic functions in $\har$, with the form \eqref{series}, satisfying $0<|b_1|<1$ and 
\begin{equation*}
\psi_1|b_1|+\sum_{n=2}^{\infty}(\varphi_n |a_n|+\psi_n |b_n|)\le 1.
\end{equation*}
Let $\har^0(\{\varphi_n\},\{\psi_n\})$ be the subclass of $\har(\{\varphi_n\},\{\psi_n\})$ satisfying $b_1=0$ in the last inequality.
In particular, Avci and Zlotkiewic \cite{AZ91} considered the special classes $\har^0(\{n\},\{n\})$ and $\har^0(\{n^2\},\{n^2\})$ for full starlikeness and full convexity, respectively. Their results can be concluded as the following theorem.

\begin{theorem}[\cite{AZ91}]\label{thm:avzl}
Let $f=h+\overline{g}$ be of the form \eqref{series} with $b_1=0$. Then
\begin{itemize}
\item[(i)] $f\in\har^0(\{n\},\{n\})$ is harmonic fully starlike and univalent in $\D$;
\item[(ii)] $f\in\har^0(\{n^2\},\{n^2\})$ is harmonic fully convex and univalent in $\D$.
\end{itemize}

\end{theorem}
\begin{rem}
A refined formulation of this theorem can be found in recent works, such as \cite{ABP14, BP14, KPV14}, which provide further insights and extensions.
\end{rem}

Silverman \cite{S98} investigated the class $\har^0(\{n\},\{n\})$ and showed the necessity of Theorem \ref{thm:avzl} if $a_n,~b_n\le0$ for $n\ge2$. Later, Jahangiri \cite{J98, J99} generalized Theorem \ref{thm:avzl} to the case that $b_1$ is not necessarily zero and proved the following result. 

\begin{theorem}[\cite{J98, J99}]\label{thm:jaja}
Let $f=h+\overline{g}$ be of the form \eqref{series} with $|b_1|<1$. Suppose that $\psi_1=1$. Then
\begin{itemize}
\item[(i)] $f\in\har(\{n\},\{n\})$ is harmonic fully starlike and univalent in $\D$;
\item[(ii)] $f\in\har(\{n^2\},\{n^2\})$ is harmonic fully convex and univalent in $\D$.
\end{itemize}
\end{theorem}

Moreover, Jahangiri \cite{J98, J99} proved the necessity of Theorem \ref{thm:jaja} if $a_n\le0$ for $n\ge2$ and $b_n\le0$ for $n\ge1$.
%
%
%
In general, Ganczar \cite{G08} presented a sufficient condition for function $f\in\har^0(\{\psi_n\},\{\psi_n\})$ which can be extended to a quasiconformal homeomorphism of $\C$. 
Hamada, et al. \cite{HHS13} then extended his results to the case $b_1$ is not necessarily $0$. Their results are summerized in the following theorem.

\begin{theorem}[\cite{G08, HHS13}]\label{thm:ghhs}
Let $\{\psi_n\}_{n=1,2,...}$ be a sequence satisfying either of the following conditions:
\begin{align}
&\frac{\psi_n}n\ge\psi_1>1\quad(n\ge1);\label{cond1}\\
&\frac{\psi_n}n\ge\frac{\psi_2}2>\psi_1=1\quad(n\ge2).\label{cond2}
\end{align}
If $f\in\har(\{\psi_n\},\{\psi_n\})$, then $f$ has a homeomorphic extension to the unit circle $\T$, and the image curve $f(\T)$ is a quasicircle. Moreover, the mapping $F$ of the form
\begin{equation}\label{eq:ghhs}
\begin{aligned}
F(z)=\left\{\begin{array}{ll}
f(z), & |z|\le1,\\
z+\sumb a_n\overline{z^{-n}}+\suma\overline{b_n}z^{-n}, & |z|\ge1,
\end{array}\right.
\end{aligned}
\end{equation}
is a quasiconformal extension of $f$ to $\C$. Furthermore, the complex dilatation $\mu_F$ satisfies $|\mu_F|\le 1/\psi_1$ if $0<|b_1|<1$ under the condition \eqref{cond1}, or $|\mu_F|\le 2/\psi_2$ if $b_1=0$ under the condition \eqref{cond2}.
\end{theorem}

A class of functions $f\in\har\seq$ with two sequences $\{\varphi_n\}$, $\{\psi_n\}$ not exactly same, is considered in Section $2$. 
As a refinement of Theorem \ref{thm:ghhs}, we claims that some functions $f\in\har\seq$ have a quasiconformal extension to the whole plane.  
It is known that any convex function $f\in\es$ has a quasiconformal extension to $\C$ if $|f|\leq C$, where $C$ is a constant (see \cite{FKZ76}). 
With a sufficient condition considered in Theorem \ref{thm:jaja}, similar result holds for harmonic convex functions.

Sufficient condition for $f\in\har$ to be harmonic strongly starlike is given as \cite[Thm.~4.2]{MPS21}.
It is shown that the mapping $f$ admits a quasiconformal extension to $\C$. 

For harmonic functions defined in the exterior of the unit disk, similar problems and corresponding results can be found in the last section. This class of such functions was first introduced by Hengartner and Schober \cite{HG87} in 1987. For more results about harmonic functions of the exterior of the unit disk, see \cite{BH94, Jun93, kalaj20}.

Our approach to quasiconformal extensions constitutes a refinement of the conceptual basis underlying conventional techniques. 
Whereas 
most existing results \cite{ABP14, AZ91, BP14, G08, HHS13, J98, J99, KPV14, S98} 
rely on a single sequence of bounds to control the dilatation uniformly, we introduce a \textit{dual-sequence framework} (as in Theorem \ref{thm:ext}) that independently constrains the coefficients for the analytic and co-analytic parts. The coefficient summation condition \eqref{cond:1} strictly weaker than the requirement in Theorem C, thus admitting broader classes of harmonic mappings.

Additionally, our proofs show bi-Lipschitz continuity with sharp constants rather than merely quasiconformality. 
These refinements collectively extend quasiconformal extension theory to a broader class of mappings, accompanied by stricter geometric constraints.

\bigskip
\section{Sufficient condition for quasiconformal extensions}

For $f\in\har\seq$, we make a refinement of Theorem \ref{thm:ghhs}, including the case that sequences $\{\varphi_n\}$, $\{\psi_n\}$ are different. We obtain the following theorem.

\begin{thm}\label{thm:ext}
For given two real numbers $k_1$, $k_2$ with $0<k_1<1$, $0<k_2<1$, let $\{\varphi_n\}$ and $\{\psi_n\}$ be two sequences of positive real numbers, which satisfy 
\begin{equation}\label{cond:1}
\begin{aligned}
&\frac{\varphi_n}{n}\ge\frac{1}{k_1}\quad(n\ge2),\\
&\frac{\psi_n}{n}\ge\frac{1}{k_2}\quad(n\ge1).
\end{aligned}
\end{equation} 
Suppose that $f\in\har\seq$. 
Then $f$ is univalent on $\D$ and has a homeomorphic extension to the unit circle. 
Moreover, the mapping $F$ of the form \eqref{eq:ghhs} is a quasiconformal extension of $f$ to $\C$, satisfying the bi-Lipschitz continuity on $\C$, with $|\mu_F(z)|\le k_2$ for $z\in\D$, and $|\mu_F(z)|\le k_1$ for $z\in\C\backslash\bD$. 
Therefore, $F$ is a $k$-quasiconformal mapping of $\C$, where $k=\max{\{k_1,k_2\}}$.
\end{thm}

\begin{proof}
Consider a function $f=h+\overline{g}\in\har\seq$ of the form \eqref{series}. If $f$ satisfies the condition \eqref{cond:1}, then we have
$$\sum_{n=2}^{\infty} n|a_n|+\sum_{n=1}^{\infty} n|b_n|\le k_1\sum_{n=2}^{\infty}\varphi_n|a_n|+k_2\sum_{n=1}^{\infty}\psi_n|b_n|\le \max{\{k_1,k_2\}}=k<1.$$
For any two points $z_1$, $z_2\in\D$ with $z_1\neq z_2$, we have
\begin{equation}\label{eq:lip1}
0<(1-k)|z_1-z_2|\le |f(z_1)-f(z_2)|\le (1+k)|z_1-z_2|.
\end{equation}
Then $f$ is univalent on $\D$ and has a homeomorphic extension to $\bD$. The image $f(\T)$ is a Jordan curve.
For $z_1$, $z_2\in\Delta$, it follows from \eqref{eq:ghhs} that
\begin{equation}\label{eq:lip2}
0<(1-k)|z_1-z_2|\le |F(z_1)-F(z_2)|\le (1+k)|z_1-z_2|.
\end{equation}
It means that $F$ is bi-Lipschitz continuous on $\Delta$, and has a homeomorphic extension to $\overline{\Delta}$. 

Next suppose $z_1\in\D$ and $z_2\in\Delta$. Let $\zeta=[z_1,z_2]\cap\T$, and $w_0=[F(z_1),F(z_2)]\cap F(\T)=F(\zeta_0)$, where $\zeta_0\in\T$. Therefore, by \eqref{eq:lip1} and \eqref{eq:lip2}, the upper and lower bounds
\begin{equation*}
|F(z_1)-F(z_2)|=|f(z_1)-f(\zeta)+F(\zeta)-F(z_2)|\le|f(z_1)-f(\zeta)|+|F(\zeta)-F(z_2)|\le (1+k)|z_1-z_2|,
\end{equation*}
and
\begin{equation*}
|F(z_1)-F(z_2)|=|F(z_1)-w_0|+|w_0-F(z_2)|\ge (1-k)(|z_1-\zeta_0|+|z_2-\zeta_0|)\ge (1-k)|z_1-z_2|.
\end{equation*}
hold, which verifies the bi-Lipschitz continuity of $F$ on $\C$.

Finally, we compute the dilatation of the mapping $F$. 
For $z\in\D$, the dilatation satisfies
\begin{equation}\label{eq:muf}
\begin{aligned}
|\mu_f(z)|=\left|\frac{\sum_{n=1}^{\infty} nb_nz^{n-1}}{1+\sum_{n=2}^{\infty} na_nz^{n-1}}\right|\le \frac{\sum_{n=1}^{\infty} n|b_n|}{1-\sum_{n=2}^{\infty} n|a_n|}
\le \frac{k_2\sum_{n=1}^{\infty} \psi_n|b_n|}{1-k_1\sum_{n=2}^{\infty} \varphi_n|a_n|}
\le k_2.
\end{aligned}
\end{equation}
The computation shows by using \eqref{eq:ghhs} that
\begin{equation}\label{eq:muF}
\begin{aligned}
|\mu_F(z)|=\left|\frac{\sum_{n=2}^{\infty} na_nz^{-n-1}}{1-\sum_{n=1}^{\infty} n\overline{b_n}z^{-n-1}}\right|\le \frac{\sum_{n=2}^{\infty} n|a_n|}{1-\sum_{n=1}^{\infty} n|b_n|}
\le \frac{k_1\sum_{n=2}^{\infty} \varphi_n|a_n|}{1-k_2\sum_{n=1}^{\infty} \psi_n|b_n|}
\le k_1,
\end{aligned}
\end{equation}
for $z\in\C\backslash\bD$. Then for any $z\in\C$, it comes to a conclusion that $|\mu_F(z)|\le\max{\{k_1,k_2\}}=k$. Therefore, $F$ is $k$-quasiconformal off $\T$ on $\C$, since the unit circle $\T$ is removable for quasiconformality. The proof is complete.

\end{proof}
With a minor change of Theorem \ref{thm:ext} we get the following result.
\begin{cor}\label{cor:ext}
Let $\{\varphi_n\}_{n=2,3,...}$ and $\{\psi_n\}_{n=1,2,...}$ be two sequences of positive real numbers, satisfying $\varphi_n\geq n$, $\psi_n\geq n$, $n=2,3,\cdots$, and $\psi_1\ge 1$. If $f\in\har\seq$ satisfies 
\begin{equation}\label{eq:hphik}
\psi_1|b_1|+\sum_{n=2}^{\infty}(\varphi_n |a_n|+\psi_n |b_n|)\le k_0<1,
\end{equation}
then the mapping \eqref{eq:ghhs} is a quasiconformal extension of $f$ to $\C$, and $|\mu_F(z)|\le k_0<1$ for $z\in\C$.
\end{cor}
\begin{pf}
Replacing the inequality \eqref{eq:hphik} with 
\begin{equation*}
\frac{\psi_1|b_1|}{k_0}+\sum_{n=2}^{\infty}\left(\frac{\varphi_n}{k_0} |a_n|+\frac{\psi_n}{k_0} |b_n|\right)\le 1,
\end{equation*}
the sequences $\{\varphi_n\}$ and $\{\psi_n\}$ satisfy \eqref{cond:1}. The proof is done.

\end{pf}

It is shown in Theorem \ref{thm:jaja} that a harmonic function $f$ in $\har(\{n^2\},\{n^2\})$ is fully convex. Since the sequences $\{\varphi_n\}$ and $\{\psi_n\}$ of $f$ satisfy the condition \eqref{cond:1}, $f$ can be extended to the plane quasiconformally.

\begin{cor}\label{cor:cvx}
Suppose that $f\in\har(\{n^2\},\{n^2\})$. Then $f$ is univalent and convex in $\D$, and has a quasiconformal extension $F$ of the form \eqref{eq:ghhs} to $\C$, with $|\mu_F(z)|\le 1/2$ for $z\in\C$.
\end{cor}

\bigskip
\section{Quasiconformal extension of harmonic strongly starlike functions}

In this section, we mainly discuss about harmonic strongly starlike functions of $\D$.
For $0<\alpha<1,$ the following quantities:
\begin{equation*}
\begin{aligned}
\varphi_n(\alpha)&=\frac{n-1+\sqrt{n^2-2n\cos{\pi\alpha}+1}}{2\sin{(\pi\alpha/2})}, \\
\psi_n(\alpha)&=\frac{n+1+\sqrt{n^2+2n\cos{\pi\alpha}+1}}{2\sin{(\pi\alpha/2})},
\end{aligned}
\end{equation*}
have been introduced first and \cite[Lemma~4.1]{MPS21} shows
\begin{equation}\label{eq:ab}
n< \varphi_n(\alpha)<\frac{n}{\sin{(\pi\alpha/2)}}<\psi_n(\alpha) 
\end{equation}
for $n\ge2$.

By using $\varphi_n(\alpha)$ and $\psi_n(\alpha)$ we restate the theorem \cite[Theorem~4.2]{MPS21} as follows.

\begin{lem}[\cite{MPS21}]\label{lem:coef}    
Let $f=h+\bar g\in\har$ for $h(z)=z+a_2 z^2+a_3 z^3+\cdots$
and $g(z)=b_1z+b_2z^2+b_3z^3+\cdots.$
Suppose that the inequality
\begin{equation}\label{eq:ci}
\sum_{n=2}^{\infty}\varphi_n(\alpha)|a_n|
+\sum_{n=1}^{\infty}\psi_n(\alpha)|b_n|\leq 1
\end{equation}
holds. Then $f\in\hsst(\alpha).$
\end{lem}

\begin{lem}\label{lem:phin}
For $0<\alpha<1$, we have
\begin{itemize}
\item[(i)] $\varphi_n(\alpha)/n$ is a strictly increasing function of $n$ with $n\ge2$;
\item[(ii)] $\psi_n(\alpha)/n$ is a strictly decreasing function of $n$ with $n\ge1$.
\end{itemize}
\end{lem}
\begin{pf}
We only give the proof of (i). 

Let $a=2\sin{(\pi\alpha/2)}$. To prove $\varphi_n(\alpha)/n$ is a strictly increasing function of $n\ge2$, we have to verify that
$$\frac{\varphi_{n+1}(\alpha)}{n+1}-\frac{\varphi_n(\alpha)}{n}=\frac{1+n\sqrt{n^2+(n+1)a^2}-(n+1)\sqrt{(n-1)^2+na^2}}{n(n+1)a}=\frac{L(n)}{n(n+1)a}$$
is positive. 
If $L(n)\le0$, then a simple computation shows $a^2\geq4$,
which is in contradiction with the fact that $a<2$. 
Hence $L(n)>0$, which means that $\varphi_n(\alpha)/n$ is a strictly increasing function of $n$.\\
\end{pf}

Using Theorem \ref{thm:ext}, Lemmas \ref{lem:coef} and \ref{lem:phin}, we obtain the following theorem.
\begin{thm}\label{thm:ssa}
Let $0<\alpha<1$. 
Then a function $f\in\har(\{\varphi_n(\alpha)\},\{\psi_n(\alpha)\})$ belongs to $\hsst(\alpha)$ and has a quasiconformal extension $F$ to $\C$.
Moreover, $F$ has the form \eqref{eq:ghhs}, with the dilatation $$|\mu_F(z)|\le\frac{\sin{(\pi\alpha/2)}}{1+\cos{(\pi\alpha/2)}}$$ for $z\in\D$, and $|\mu_F(z)|\le\sin{(\pi\alpha/2)}$ for $z\in\C\backslash\bD$. 
Furthermore, $f$ is $\operatorname{a} $ $\sin{(\pi\alpha/2)}$-quasiconformal mapping of the whole plane. 
\end{thm}

\begin{pf}
It is known that $\varphi_n(\alpha)/n$ is a strictly increasing function of $n\ge2$, and $\psi_n(\alpha)/n$ is a strictly decreasing function of $n\ge1$, by Lemma \ref{lem:phin}.
The inequality \eqref{eq:ab} shows that the two functions have no intersection, and both approach to $1/\sin{(\pi\alpha/2)}$ if $n$ tends to infinity.
Following Lemma \ref{lem:coef}, the function $f\in\har(\{\varphi_n(\alpha)\},\{\psi_n(\alpha)\})$ belongs to the class $\hsst(\alpha)$, for $0<\alpha<1$.
Now applying Theorem \ref{thm:ext} to $f$, we obtain
\begin{equation*}
\begin{aligned}
|\mu_F(z)|\le k
& = \max\bigg
\{\frac{1}{\psi_1(\alpha)},\sin{\Big(\frac{\pi\alpha}{2}\Big)}
\bigg\}\\
& = \max\bigg\{\frac{\sin{(\pi\alpha/2)}}{1+\cos{(\pi\alpha/2)}},\sin{\Big(\frac{\pi\alpha}{2}\Big)}\bigg\}\\
& = \sin{\Big(\frac{\pi\alpha}{2}\Big)}< 1,
\end{aligned}
\end{equation*}
for $z\in\C$, by \eqref{eq:muf} and \eqref{eq:muF}. Then the function $F$ is a $\sin{(\pi\alpha/2)}$-quasiconformal mapping of $\C$.

\end{pf}

A simple example of harmonic strongly starlike functions of order $\alpha$ constructed in \cite{MPS21}, gives a sharp bound for the second coefficient of the co-analytic part. Here, by Theorem \ref{thm:ssa}, an explicit form of the extension functions for more general functions is given. 

\begin{eg}
For $\alpha$ with $0<\alpha<1$, we consider the function
$$f_n(z)=z+b_n\overline{z}^{n},\quad n\ge2,$$
where $$|b_n|\le\frac{1}{\psi_n(\alpha)}=\frac{2\sin{(\pi\alpha/2)}}{(n+1)+|n+e^{i\pi\alpha}|}.$$
The coefficients of $f_n$ satisfy the condition \eqref{eq:ci}, so $f_n\in\hsst(\alpha)$. Using Theorem \ref{thm:ssa}, $f_n$ can be extended to $\C$, and the mapping of the form 
\begin{equation*}
\begin{aligned}
F_n(z)=\left\{\begin{array}{ll}
z+b_n\overline{z}^n, & |z|\le1,\\
z+b_nz^{-n}, & |z|\ge1,
\end{array}\right.
\end{aligned}
\end{equation*}
is a quasiconformal extension of $f_n$, with the dilatation $$|\mu_{F_n}(z)|\le\frac{2n\sin{(\pi\alpha/2)}}{n+1+|n+e^{i\pi\alpha}|}$$
for $z\in\C$.
\end{eg}

\section{Quasiconformal extension of harmonic mappings in the exterior of the unit disk}

In this section, we investigate the class of harmonic sense-preserving univalent functions defined in the exterior of the unit disk $\Delta=\{z:|z|>1\}$ that map $\infty$ to $\infty$. 
This class was initiated by Hengartner and Schober \cite{HG87}, denoted by $\Sig$. 
Each function $f\in\Sig$ has the representation 
\begin{equation}\label{eq:sig0}
f(z)=\alpha z+\beta\overline{z}+\sum_{n=0}^{\infty}a_n z^{-n}+\overline{\suma b_n z^{-n}}+A\log{|z|},\quad z\in\Delta,
\end{equation}
where $0\le|\beta|<|\alpha|$ and $A\in\C$.
With the assumption $\alpha=1$, $\beta=0$ and $a_0=0$, a subclass of $\Sig$, denoted by $\Siga$, consisting of functions $f$ of the form 
$$f(z)=z+h(z)+\overline{g(z)}+A\log{|z|},\quad z\in\Delta,$$
where
$$h(z)=\suma a_n z^{-n},\quad g(z)=\suma b_n z^{-n}$$
are analytic in $\Delta$, has been considered. 
For convenience, a subclass of $\Siga$ with $A=0$ is needed, denoted by $\Sigb$ (see \cite{J00}). 

For function $f\in\Sig$ of the form \eqref{eq:sig0}, a simple computation shows that a sufficient condition for $|\mu_f|\le k<1$ is 
\begin{equation}\label{eq:ck0}
|\beta|+\frac{k+1}2|A|+k\suma n|a_n|+\suma n|b_n|\le k|\alpha|.
\end{equation}
To obtain a stronger form of \eqref{eq:ck0}, it suffices to assume
\begin{equation}\label{eq:ck1}
|\beta|+|A|+\suma n(|a_n|+|b_n|)\le k|\alpha|.
\end{equation}
In search of a sufficient condition for the quasiconformal extendibility of functions $f\in\Sig$, a subclass of $\Sig$ which satisfies \eqref{eq:ck1} is required to be considered, denoted by $\Sig(k)$, for $0<k<1$.
A theorem is then stated as follows, as a generalization of \cite[Thm.~9]{WG16} for functions $f\in\Sigb$.

\begin{thm}\label{thm:sig1k}
Let $f\in\Sig(k)$ be of the form \eqref{eq:sig0} for some $k\in(0,1)$.
Then $f$ has a homeomorphic extension to the unit circle. Moreover, the mapping
\begin{equation}\label{eq:sig1ext}
\begin{aligned}
F(z)=\left\{\begin{array}{ll}
f(z), & |z|\ge1,\\
\alpha z+\beta\overline{z}+\sum_{n=0}^{\infty} a_n\overline{z}^{n}+\suma\overline{b_n}z^{n}, & |z|\le1,
\end{array}\right.
\end{aligned}
\end{equation}
is a quasiconformal extension of $f$ with $|\mu_F(z)|\le k$ for $z\in\C$.
\end{thm}

\begin{pf}
Let $f\in\Sig(k)$ take the form \eqref{eq:sig0}. For any points $z_1$, $z_2$ in $\Delta$ with $z_1\neq z_2$, it is harmless to assume $|z_1|\ge|z_2|>1$. We compute
\begin{equation*}
\begin{aligned}
&  |f(z_1)-f(z_2)|\\
&= \left|\alpha(z_1-z_2)+\beta(\overline{z_1-z_2})+\sum_{n=0}^{\infty} a_n(z_1^{-n}-z_2^{-n})+\overline{\suma b_n(z_1^{-n}-z_2^{-n})}+A(\log{|z_1|}-\log{|z_2|})\right|\\
&\ge |z_1-z_2|(|\alpha|-|\beta|)-\suma (|a_n|+|b_n|)|z_1^{-n}-z_2^{-n}|-|A|(\log{|z_1|}-\log{|z_2|}) \\
&\ge |z_1-z_2|\left(|\alpha|-|\beta|-\suma n(|a_n|+|b_n|)\right)-|A|\int_{|z_2|}^{|z_1|}dt\\
&\ge |z_1-z_2|(1-k)|\alpha|.
\end{aligned}
\end{equation*}
Similarly, we obtain
\begin{equation*}
|f(z_1)-f(z_2)|\le|z_1-z_2|(1+k)|\alpha|.
\end{equation*}
Therefore, $f$ satisfies the bi-Lipschitz condition
\begin{equation*}
0<(1-k)|\alpha|\cdot|z_1-z_2|\le|f(z_1)-f(z_2)|\le(1+k)|\alpha|\cdot|z_1-z_2|,\quad z\in\Delta.
\end{equation*}
It means that $f$ has a homeomorphic extension to the unit circle $\T$, and $f(\T)$ is a quasicircle.

Next it is necessary to show that the function $F$ of the form \eqref{eq:sig1ext} is a $k$-quasiconformal mapping of the plane.
The dilatation satisfies
\begin{equation*}
\begin{aligned}
|\mu_f(z)|=\left|\frac{\beta+\suma na_n\overline{z}^{n-1}}{\alpha+\suma n\overline{b_n}z^{n-1}}\right|\le \frac{|\beta|+\suma n|a_n|}{|\alpha|-\suma n|b_n|}\le k,
\end{aligned}
\end{equation*}
for $z\in\D$ and
\begin{equation*}
\begin{aligned}
|\mu_F(z)|=\left|\frac{\beta+\frac{A}{2\overline{z}}-\overline{\suma nb_nz^{-n-1}}}{\alpha+\frac{A}{2z}-\suma na_nz^{-n-1}}\right|
\le\frac{|\beta|+\frac{|A|}2+\suma n|b_n|}{|\alpha|-\frac{|A|}2-\suma n|a_n|}
\le\frac{2k|\alpha|-|A|}{2|\alpha|-|A|}\le k
\end{aligned}
\end{equation*}
by \eqref{eq:sig1ext}, for $z\in\Delta$.
Hence, $|\mu_F(z)|\le k$ for $z\in\C$. The proof is complete.

\end{pf}

\begin{eg}
Consider the function $$f(z)=z-\frac{i}6\overline{z}+\frac{i}4\log{|z|}-\frac{i}{8}z^{-4},$$ which belongs to $\Sig$. Since $f$ satisfies \eqref{eq:ck1}, we apply Theorem \ref{thm:sig1k} to $f$. Then $f$ has a homeomorphic extension to the unit circle, and the mapping 
\begin{equation*}
\begin{aligned}
F(z)=\left\{\begin{array}{ll}
z-\frac{i}6\overline{z}+\frac{i}4\log{|z|}-\frac{i}{8}z^{-4}, & |z|\ge1,\\
z-\frac{i}6\overline{z}-\frac{i}8\overline{z}^4, & |z|\le1,
\end{array}\right.
\end{aligned}
\end{equation*}
is a $k$-quasiconformal extension of $f$ with $k=7/9$. See Figure \ref{fig:1} for the graph of $F$.
\begin{figure}
\begin{center}
\includegraphics[width=.6\textwidth]{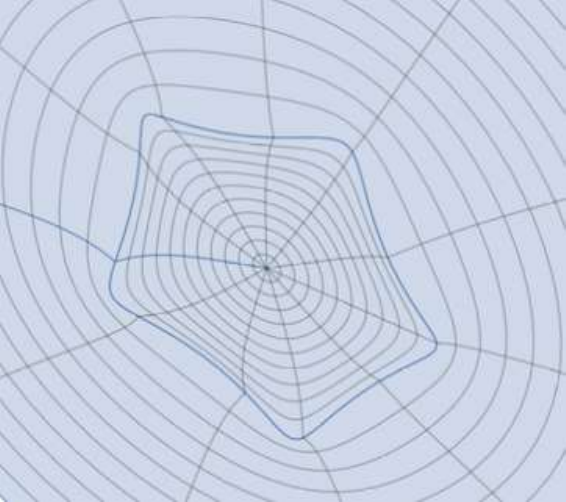}
\caption{The graph of $F$}\label{fig:1}
\end{center}
\end{figure}
\end{eg}
\bigskip
Before giving the next theorem, we introduce the convolution of harmonic functions defined in $\Delta$. 
Let $f_1(z),~f_2(z)\in\Sig$ be two harmonic functions in $\Delta$, with the forms
\begin{equation}\label{eq:s4f1}
f_1(z)=\alpha_1z+\beta_1\overline{z}+\sumo a_nz^{-n}+\overline{\suma b_nz^{-n}}+c\log{|z|},
\end{equation}
and 
\begin{equation}\label{eq:s4f2}
f_2(z)=\alpha_2z+\beta_2\overline{z}+\sumo A_nz^{-n}+\overline{\suma B_nz^{-n}}+C\log{|z|},
\end{equation}
for $z\in\Delta$. Then the {\it harmonic convolution} of $f_1$ and $f_2$ is
\begin{equation}\label{eq:cvl}
f_1*f_2(z)=\alpha_1\alpha_2z+\beta_1\beta_2\overline{z}+\sumo a_nA_nz^{-n}+\overline{\suma b_nB_nz^{-n}}+cC\log{|z|},\quad z\in\Delta.
\end{equation}

Let $\Sigma_k$ $(0<k<1)$ be the class of sense-preserving homeomorphisms $h$ of the extended plane $\sphere$ onto itself, with $h(z)=z+\sum_{n=0}^{\infty} a_nz^{-n}$ analytic and univalent in $\Delta$, $k$-quasiconformal in $\sphere$. 
Krzyż \cite{K76} investigated the convolution problem of functions in $\Sigma_k$ by using the area theorem. 
Unfortunately, the area theorem of harmonic functions (see \cite{M04}) cannot be used in the same way.
However, we deduce the following theorem.

\begin{thm}\label{thm:cvl}
Let $0<k_1,~k_2<1$. If $f_1\in\Sig(k_1)$ and $f_2\in\Sig(k_2)$, then $f_1*f_2\in\Sig(\sqrt{k_1k_2})$.
\end{thm}

\begin{pf}
Let $f_1$ and $f_2$ be functions defined by \eqref{eq:s4f1} and \eqref{eq:s4f2}, respectively.
Given that $f_1\in\Sig(k_1)$ and $f_2\in\Sig(k_2)$, it follows immediately from the definition of the class $\Sig$ that their convolution $f_1*f_2$ also belongs to $\Sig$.
Condition \eqref{eq:ck1} yields the coefficient bounds
$$
|\beta_1|+|c|+\suma n(|a_n|+|b_n|)\le k_1|\alpha_1|,
$$ and 
$$
|\beta_2|+|C|+\suma n(|A_n|+|B_n|)\le k_2|\alpha_2|.
$$
It suffices to show 
$$
M=\frac{1}{|\alpha_1\alpha_2|}\left[|\beta_1\beta_2|+\suma n(|a_nA_n|+|b_nB_n|)+|cC|\right]\le\sqrt{k_1k_2}.$$
Define sequences $\{x_m\}$ and $\{X_m\}$ as
\begin{equation*}
\begin{aligned}
|\alpha_1|x_m=
\left\{\begin{array}{ll}
\sqrt{n|a_n|^2}, & m=2n-1,\\
\sqrt{n|b_n|^2}, & m=2n,\\
|\beta_1|+|c|, & m=0,
\end{array}\right. \ and\quad
|\alpha_2|X_m=\left\{\begin{array}{ll}
\sqrt{n|A_n|^2}, & m=2n-1,\\
\sqrt{n|B_n|^2}, & m=2n,\\
|\beta_2|+|C|, & m=0.
\end{array}\right.
\end{aligned}
\end{equation*}
Then the quantity $M$ satisfies
$$M\le\sum_{m=0}^{\infty}x_mX_m.$$ 
Finally we obtain
\begin{equation*}
\begin{aligned}
M & \le \left(\sum_{m=0}^{\infty} x_m^2\right)^{\frac{1}2}\left(\sum_{m=0}^{\infty} X_m^2\right)^{\frac{1}2}\\
& = \left(\frac{(|\beta_1|+|c|)^2+\suma n(|a_n|^2+|b_n|^2)}{|\alpha_1|^2}\right)^{\frac{1}2}\left(\frac{(|\beta_2|+|C|)^2+\suma n(|A_n|^2+|B_n|^2)}{|\alpha_2|^2}\right)^{\frac{1}2}\\
& \le \left(\frac{(|\beta_1|+|c|)+\suma n(|a_n|+|b_n|)}{|\alpha_1|}\right)^{\frac{1}2}\left(\frac{(|\beta_2|+|C|)+\suma n(|A_n|+|B_n|)}{|\alpha_2|}\right)^{\frac{1}2}\\
& \le \sqrt{k_1k_2}.
\end{aligned}
\end{equation*}
by the Cauchy-Schwarz inequality.

Thus $f_1*f_2\in\Sig(\sqrt{k_1k_2})$, and the proof is complete.

\end{pf}


\section*{Acknowledgements}
    The author would like to express sincere gratitude to Prof. Toshiyuki Sugawa, for his expertise and patience he has provided during the writing of this paper. 
    The author would also like to offer special thanks to Prof. Saminathan Ponnusamy, Prof. Limei Wang and Prof. Ming Li for their valuable and constructive suggestions.

  \addcontentsline{toc}{chapter}{Acknowledgements}



\begin{thebibliography}{10}

\def\cprime{$'$} \def\cprime{$'$} \def\cprime{$'$}
\providecommand{\bysame}{\leavevmode\hbox to3em{\hrulefill}\thinspace}
\providecommand{\MR}{\relax\ifhmode\unskip\space\fi MR }
\providecommand{\MRhref}[2]{%
  \href{http://www.ams.org/mathscinet-getitem?mr=#1}{#2}
}
\providecommand{\href}[2]{#2}

\bibitem{ABP14}
Y.~Abu~Muhanna, S.~V.~Bharanedhar and S.~Ponnusamy, \emph{One parameter family of univalent biharmonic mappings}, Taiwanese J. Math. \textbf{18} (2014), no. 4, 1151-1169.

\bibitem{ahlfors06}
L.~V.~Ahlfors, \emph{Lectures on Quasiconformal Mappings}, American Mathematical Soc. \textbf{38}, 2006.

\bibitem{AZ91}
Y.~Avci and E.~Z{\l}otkiewicz, \emph{On harmonic univalent mappings}, Ann. Univ. Mariae Curie-Sk{\l}odowska Sect. A \textbf{44} (1991), 1--7.

\bibitem{BP14}
S.~V.~Bharanedhar and S.~Ponnusamy. \emph{Coefficient conditions for harmonic univalent mappings and hypergeometric mappings}, Rocky Mountain J. Math. \textbf{44} (2014), no. 3, 753--777.

\bibitem{BH94}
D.~Bshouty, and W.~Hengartner. \emph{Univalent harmonic mappings in the plane}, Ann. Univ. Mariae Curie-Sk{\l}odowska Sect. A \textbf{48} (1994), no. 3, 12--42.

\bibitem{CDO04}
M.~Chuaqui, P.~Duren and B.~Osgood, \emph{Curvature properties of planar harmonic mappings}, Comput. Methods Funct. Theory \textbf{4} (2004), no. 1, 127--142.

\bibitem{CS84}
J.~Clunie and T.~Sheil-Small, \emph{Harmonic univalent functions}, Ann. Acad. Sci. Fenn. Ser. A I Math. \textbf{9} (1984), 3--25. 

\bibitem{D04}
P.~Duren, \emph{Harmonic Mappings in the Plane}, Cambridge university press \textbf{156}, 2004.

\bibitem{FKZ76}
M.~Fait, J.~G.~Krzyż and J.~Zygmunt, \emph{Explicit quasiconformal extensions for some classes of univalent functions}, Comment. Math. Helv. \textbf{51} (1976), no. 1, 279--285.

\bibitem{G08}
A.~Ganczar, \emph{Explicit quasiconformal extensions of planar harmonic mappings}, J. Comput. Anal. Appl. \textbf{10} (2008), no. 2, 179--186.

\bibitem{HHS13}
H.~Hamada, T.~Honda and K.~H.~Shon, \emph{Quasiconformal extensions of starlike harmonic mappings in the unit disc}, Bull. Korean Math. Soc. \textbf{50} (2013), no. 4, 1377--1387.

\bibitem{HG87}
W.~Hengartner and G.~Schober, \emph{Univalent harmonic functions}, Trans. Amer. Math. Soc. \textbf{299} (1987), no. 1, 1--31.

\bibitem{J98}
J.~M.~Jahangiri, \emph{Coefficient bounds and univalence criteria for harmonic functions with negative coefficients}, Ann. Univ. Mariae Curie-Sk{\l}odowska Sect. A \textbf{52} (1998), 57--66.

\bibitem{J99}
J.~M.~Jahangiri, \emph{Harmonic functions starlike in the unit disk}, J. Math. Anal. Appl. \textbf{235} (1999), no. 2, 470--477.

\bibitem{J00}
J.~M.~Jahangiri, \emph{Harmonic meromorphic starlike functions}, Bull. Korean Math. Soc. \textbf{37} (2000), no. 2, 291--301.

\bibitem{Jun93}
S.~H.~Jun, \emph{Univalent harmonic mappings on $\Delta=\{z: |z|>1\}$}, Proc. Amer. Math. Soc. \textbf{119} (1993), no.1, 109--114.

\bibitem{kalaj20}
D.~Kalaj, \emph{Quasiconformal harmonic mappings between domains containing infinity}, Bull. Aust. Math. Soc., \textbf{102} (2020), no.1, 109--117.

\bibitem{KPV14}
D.~Kalaj, S.~Ponnusamy and M.~Vuorinen, \emph{Radius of close-to-convexity and fully starlikeness of harmonic mappings}, Complex Var. Elliptic Equ., \textbf{59} (2014), no.4, 539-552. 

\bibitem{K76}
J.~G.~Krzyż, \emph{Convolution and quasiconformal extension}, Comment. Math. Helv. \textbf{51} (1976), no. 1, 99--104.

\bibitem{LV73}
O.~Lehto and K.~I.~Virtanen, \emph{Quasiconformal Mappings in the Plane}, \textbf{126}, Springer-Verlag, New York, 1973.


\bibitem{MPS21}
X.-S.~Ma, S.~Ponnusamy and T.~Sugawa, \emph{Harmonic spirallike functions and harmonic strongly starlike functions}, Monatsh. Math., \textbf{199} (2022), no. 2, 363--375.

\bibitem{M04}
M.~S.~Mateljević, \emph{Dirichlets principle, distortion and related problems for harmonic mappings}, Publ. Inst. Math. (Beograd) (N.S.) \textbf{75} (2004), no. 89, 147--171.

\bibitem{S98}
H.~Silverman, \emph{Harmonic univalent functions with negative coefficients}, J. Math. Anal. Appl. \textbf{220} (1998), no. 1, 283--289.

\bibitem{WG16}
J.~Widomski and M.~Gregorczyk, \emph{Harmonic mappings in the exterior of the unit disk}, Ann. Univ. Mariae Curie-Sk{\l}odowska Sect. A \textbf{64} (2016), no. 1, 63--73.

\end{thebibliography}
\end{document}